\documentclass[12pt]{amsart}
\usepackage{fullpage}
\usepackage{amsmath,amsthm,amsfonts,amssymb,verbatim,eucal}
\usepackage[all]{xy}

\numberwithin{equation}{section}

\newtheorem{thm}[equation]{Theorem}

\newtheorem{example}[equation]{Example}
\theoremstyle{definition}
\newtheorem{remark}[equation]{Remark}

\newenvironment{rem}{\begin{remark}\rm}{\end{remark}}

\DeclareMathOperator{\Ext}{Ext}

\DeclareMathOperator{\Proj}{Proj}

\newcommand{\DOT}{*}
\newcommand{\coh}{{\rm H}}
\newcommand{\Z}{{\mathbb Z}}

\newcommand{\Hom}{\mbox{\rm Hom\,}}
\newcommand{\ra}{\rangle}
\newcommand{\la}{\langle}
\newcommand{\ot}{\otimes}
\newcommand{\p}{{\mathfrak {p}}}
\newcommand{\HH}{{\rm HH}}
\newcommand{\Mod}{\mathsf{Mod}}
\newcommand{\mmod}{\mathsf{mod}}
\newcommand{\V}{{\mathcal{V}}}
\newcommand{\StMod}{\mathsf{StMod}}    
\newcommand{\stmod}{\mathsf{stmod}}

\title[Examples of support varieties]
{Examples of support varieties for Hopf algebras
with noncommutative tensor products}

\author{Dave Benson}
\address{Institute of Mathematics, University of Aberdeen, King's College, 
Aberdeen AB24 3UE, Scotland, U.K.}

\author{Sarah Witherspoon}
\address{Department of Mathematics, Texas A\&M University, College Station,
Texas 77843, U.S.A.}

\date{22 August 2013}

\thanks{This material is based upon work supported by the National Science 
Foundation under Grant No.~0932078~000, while the authors were in 
residence at the Mathematical Science Research Institute (MSRI) in 
Berkeley, California, during the Spring semester of 2013.
The second author was supported by National Science Foundation Grant No.\
DMS-1101399.}

\begin{document}

\maketitle

\begin{abstract}
The representations of some Hopf algebras have  curious behavior: 
Nonprojective modules may have projective tensor powers, 
and the variety of a tensor product of modules
may not be  contained in the intersection of their varieties. 
We explain  a family of examples of such  Hopf algebras and their modules, 
and classify  left, right, and two-sided ideals in 
their stable module categories. 
\end{abstract}

\section{Introduction}

Over the past few decades,
the theory of support varieties has become one of the cornerstones of
the representation theory of finite groups. Its success has led to the
introduction of similar methods in the representation theory of restricted Lie algebras, 
finite group schemes, and complete intersections,
as well as in  various other branches of representation theory.

One of the main points of the theory is that projectivity 
(or, in some situations, finite projective dimension)
of a module can be detected through the support variety.
Complexity, a measure of the rate of growth of a
minimal resolution, may also be  detected.
In categories with tensor products, often the support of
a tensor product is the intersection of the supports,
allowing one to tell whether a tensor product of modules
is projective, and more generally to compute its complexity.

The purpose of this paper is to provide a family of examples of
finite dimensional Hopf algebras $A$ which are neither commutative
nor cocommutative, and for which the tensor product of modules
exhibits some unusual behaviour. Here are the main features of
this family of examples.

\begin{itemize}
\item If $M$ and $N$ are $A$-modules then 
$M\otimes N$ and $N \otimes M$ are generally not isomorphic.
\item The support of $M \otimes N$ need not be contained in
the intersection of the supports of $M$ and $N$.
\item One of the two modules $M\ot N$ and $N\ot M$ can be projective
while the other is not. 
\item There are examples of modules
$M$ which are not projective, but $M\otimes M$ is projective.
\item More generally, given $n>0$ there exists such an $A$ and $M$
with $M^{\otimes(n-1)}$ not projective but $M^{\otimes n}$ projective.
In these examples, the complexity of $M^{\otimes i}$ is equal
to $n-i$ for $1\le i\le n$; but other decreasing sequences of
complexities can be arranged at will.
\item Nonetheless there is a formula for the support 
of $M \otimes N$ in terms of the support of $M$ and of $N$.
\end{itemize}
Using the support, we classify the localising subcategories,
the left ideals, the right ideals and the two-sided ideals of
the stable module category $\StMod(A)$.

In more detail, let $G$ and $L$ be finite groups, together with
an action $L\to\mathop{\rm Aut}(G)$ of $L$ on $G$ by group automorphisms. 
Let $k$ be an algebraically closed field of characteristic $p$
dividing the order of $G$.
As an algebra, $A = kG \otimes k[L]$, the tensor product of the
group algebra of $G$ and the coordinate ring of $L$. The
coalgebra structure is dual to the algebra structure of the smash product 
$k[G]\# kL$ using the action of $L$ on $G$.

Since $k[L]$ is semisimple, its indecomposable representations are all 
one-dimensional, and correspond to the elements of $L$. We
write $k_\ell$ for the one dimensional representation corresponding
to $\ell\in L$. So every left $A$-module $M$ has a 
canonical decomposition
\[ M = \bigoplus_{\ell \in L} M_\ell \otimes k_\ell,\]
where $\ot = \ot_k$ and 
each $M_\ell$ is a $kG$-module. 

In terms of this decomposition,
the tensor product of modules is given by the following formula,
as we show in Section~\ref{GL}:
\[ (M \otimes N)_\ell = \bigoplus_{\ell_1\ell_2=\ell} M_{\ell_1} \otimes {^{\ell_1}}\!N_{\ell_2}, \]
where ${^{\ell_1}}\!N_{\ell_2}$ denotes the conjugate of the $kG$-module $N_{\ell_2}$ by
the action of the element $\ell_1$ on $G$. Similarly, the dual of
an $A$-module $M$ is shown in Section~\ref{GL} to be  given by
\[ (M^*)_\ell = {^{\ell^{-1}}}\!(M_{\ell^{-1}})^*. \]
These formulas lead to  
the examples in Section~\ref{support} of $A$-modules whose
tensor products have such curious behavior in comparison to better-known settings. 
In Section~\ref{localising}
we classify ideal subcategories of the stable module category
of all $A$-modules.

\section{A noncommutative tensor product}\label{GL}

We give first the Hopf algebra structure of $A$ explicitly,
and then some immediate consequences for its representations. 
The action of $L$ on $G$ induces an action 
on  $k[G]:= \Hom_k (kG,k)$. 
Let  $k[G]\# kL$ (or $k[G]\rtimes L$)
denote the resulting smash (or semidirect) product:
As a vector space, it is $k[G]\ot kL$, and multiplication is given by
$(p_g\ot x) (p_h\ot y) =  \delta_{g, {^x\! h}} \ p_g \ot xy$
for all $g,h\in G$ and $x,y\in L$, where $\{p_g \mid g\in G\}$ is
the basis of $k[G]$ dual to $G$, $\delta$ denotes the Kronecker delta,
and $ ^x\! h$ is the image of $h$ under the action of $x$. 
The algebra $k[G]\# kL$
is a Hopf algebra with the tensor product coalgebra structure. Let
$$
  A   := \Hom_k (k[G]\# kL , k),
$$
the Hopf algebra dual to $k[G]\# kL$.
Then $A$ has  the tensor product algebra structure, $A = kG\ot k[L]$.
The comultiplication, dual to multiplication on $k[G]\# kL$, is given by
$$
  \Delta(g\ot p_{\ell}) = \sum _{x\in L} (g\ot p_x)\ot ( {^{x^{-1}}\! g}\ot
   p_{x^{-1}\ell})
$$
for all $g\in G$, $\ell\in L$. 
The counit and coinverse are  determined by
$$
   \varepsilon(g\ot p_{\ell})= \delta_{1,\ell} \ \ \mbox{ and } \ \ 
   S(g\ot p_{\ell})=  \ ^{\ell^{-1}}\! (g^{-1})\ot p_{\ell^{-1}}.
$$
Note that $S^2$ is the identity map on $A$.

If $M,N$ are $A$-modules, then $M\ot N$ is an $A$-module via the
coproduct $\Delta$. The $k$-linear dual $M^*:=\Hom_k(M,k)$ 
is an $A$-module via the coinverse $S$:
$\ (a\cdot f) (m) =  f(S(a)m)$ for all $a\in A$, $f\in M^*$, $m\in M$.

\begin{thm}\label{tp-formula}
Let $M,N$ be $A$-modules, and $\ell \in L$. Then 
\[ (M\ot N)_{\ell} = \bigoplus_{ \ell_1\ell_2=\ell} M_{\ell_1}\ot {^{\ell_1}\! N_{\ell_2}}\]
and $(M^*)_{\ell}= {^{\ell^{-1}}\! (M_{\ell^{-1}})}^*$. 
\end{thm}

\begin{proof}
It suffices to prove the statement componentwise. Let $y,z\in L$. 
Define a $k$-linear map $\phi: (M_y\ot k_y)\ot (N_z\ot k_z)\rightarrow
(M_y\ot  {^y\! N_z})\ot k_{yz}$ by
$$
   \phi( (m\ot p_y)\ot (n\ot p_z)) = (m\ot n) \ot p_{yz}
$$
for all $m\in M_y$, $n\in N_z$.
Clearly $\phi$ is a bijection. It is also an $A$-module homomorphism:
Let $g\in G$, $\ell\in L$.
Applying the coproduct $\Delta$ to $g\ot p_{\ell}$, we find 
\begin{eqnarray*}
\phi((g\ot p_{\ell})((m\ot p_y)\ot (n\ot p_z)) & = &
    \delta_{\ell, yz} \ \phi((gm\ot p_y)\ot ( ({^{y^{-1}}\! g})n\ot p_z)\\
   &=& \delta_{\ell, yz} \ (gm\ot ({^{y^{-1}}\! g})n)\ot p_{yz} .
\end{eqnarray*}
On the other hand,
\begin{eqnarray*}
(g\ot p_{\ell})\phi((m\ot p_y)\ot (n\ot p_z) ) &=&
   (g\ot p_{\ell}) ((m\ot n)\ot p_{yz})\\
   &=& \delta_{\ell, yz} \ (gm\ot ({^{y^{-1}}\! g})n)\ot p_{yz}.
\end{eqnarray*}
 
Let $\psi:  {^{y^{-1}}\! (M_{y^{-1}})^*} \ot k_{y^{-1}} \rightarrow (M_y\ot k_y)^*$
be the $k$-linear map defined by
$$
   \psi( f\ot p_{y^{-1}}) =  \widetilde{ f}
$$
where $\widetilde{ f} (m\ot p_y) = f( m)$ for all
$m\in M_y$ and $y\in L$.
Clearly $\psi$ is a bijection. It is also an $A$-module homomorphism:
Let $g\in G$ and $\ell\in L$. Then
$$
  \psi( (g\ot p_{\ell})(f\ot p_{y^{-1}})) \ = \ \delta_{\ell , y^{-1}} \ 
  \psi( {^{^{y}g}}f\ot p_{y^{-1}}) \ = \ \delta_{\ell, y^{-1}} \ \widetilde{
      {^{{^y\! g}} f} }
$$
where $\widetilde{ {^{ {^y\! g}}\! f}} (m\ot p_y) = {^{{^y\! g}}\! f}(m) = 
f(({^y\! g})^{-1}m)$.
On the other hand,
$$
  (g\ot p_{\ell}) (\psi(f\ot p_{y^{-1}})) \ = \ (g\ot p_{\ell})( 
   \widetilde{ f})
$$
where
\begin{eqnarray*}
    ((g\ot p_{\ell})( \widetilde{ f})) (m\ot p_y) &=&
   \sum_{x\in L} (g\ot p_x) ( \widetilde{  f}(S( {^{x^{-1}}\! g}\ot 
    p_{x^{-1}\ell}) (m\ot p_y)))\\
      &=& (g\ot p_{y\ell}) \widetilde{  f}( {^{\ell^{-1}}\! (g^{-1})}m
   \ot p_y)
    \ \ = \ \        \delta_{\ell,y^{-1}} f( ({^y\! g})^{-1}m).
\qedhere
\end{eqnarray*}
\end{proof}

\section{Support varieties of finite dimensional modules}\label{support}

We recall a definition of  support varieties for finite dimensional modules,
adapted from Snashall and Solberg~\cite{SS}. 
(A more general definition, suitable for infinite dimensional modules, 
will be recalled in the next section.) 
We illustrate with examples some ways in which 
their behavior is different from  the cocommutative case. 
For this purpose, {\em assume from now on that $G$
is a $p$-group}.
(The more general case is slightly more complicated.) 
Recall that the group cohomology ring, 
$H^*(G,k):=\Ext^*_{kG}(k,k)$, may be regarded as a quotient of
the Hochschild cohomology ring, $\HH^*(kG):=\Ext^*_{kG\ot (kG)^{op}}
(kG,kG)$, by a nilpotent ideal, and
the quotient map is split by the canonical 
inclusion $H^*(G,k)\to\HH^*(kG)$ (see, e.g., \cite[Theorem 10.1]{Siegel/Witherspoon1999}).
It follows that the tensor product 
$H^*(G,k) \otimes k[L]$ may also be regarded as a quotient of
$\HH^*(A)\cong \HH^*(kG) \otimes k[L]$
by a nilpotent ideal, and the quotient map is split by the canonical 
inclusion $H^*(G,k)\otimes k[L]\to\HH^*(A)$.

For finite dimensional $A$-modules $M$ and $N$, 
consider the action of the  Hochschild cohomology ring
$\HH^*(A)$  on $\Ext_A^{\DOT}(M,N)$ given by $-\ot_A M$ followed by Yoneda composition.
By restriction we obtain an action of $\coh^*(G,k)\ot k[L]$, and it is this
action that we choose to define support varieties. 
Let $I_A(M)$ denote the annihilator in $\coh^*(G,k)\ot k[L]$ of this action
on $\Ext^*_A(M,M)$.
The {\em support variety} of $M$ is 
$$
   V_A(M) := \Proj ( \coh^*(G,k)\ot k[L] /I_A(M)),
$$
where $\Proj$ denotes the space of 
homogeneous prime ideals other than the maximal ideal of positive
degree elements. 
These varieties have standard properties as proved in \cite{EHSST,SS}.
For example, 
a finite dimensional  $A$-module $M$ is projective if, and only if, its support
variety is empty, and more generally the 
dimension  of $V_A(M)$ is the complexity of $M$ (the rate
of growth of a minimal projective resolution).

\begin{rem}
Our definition of support variety is equivalent to that of Snashall
and Solberg \cite{SS} since we have assumed $G$ is a $p$-group. 
It differs from that of Feldvoss and the second author \cite{FW},
since the cohomology $\coh^*(A,k):= \Ext^*_A(k,k)$ is isomorphic to $\coh^*(G,k)$,
and not to $\coh^*(G,k)\ot k[L]$.
It  has an advantage over
the latter in that it remembers more information about an $A$-module. 
\end{rem} 

The support variety $V_A := V_A(k)$ has the form 
\[ \Proj (H^*(G,k) \otimes k[L])=V_G \times L. \]
We write $V_{G,\ell}(M)$ for $V_G(M_\ell)$, the support variety of the $kG$-module $M_{\ell}$. Then 
\[ V_A(M)=\coprod_{\ell\in L}V_{G,\ell}(M) \times \ell \subseteq V_G\times L. \]
For each $\ell$, the variety $V_G(M_\ell)$ is the collection of primes 
containing the kernel of the map from
$H^*(G,k)$ to  $\Ext^*_{kG}(M_\ell,M_\ell)$ given by $ - \ot M_{\ell}$.

Our formula for the support of a tensor product follows directly
from the tensor product formula of Theorem~\ref{tp-formula} for
$A$-modules and properties
of support varieties for $kG$-modules:
\[ V_{G,\ell}(M \otimes N) =
\bigcup_{\ell_1\ell_2=\ell}V_{G,\ell_1}(M)\cap {^{\ell_1}}\!V_{G,\ell_2}(N). \]

\begin{example}\label{klein}
{\em 
Let $G$ be the Klein 4-group with nonidentity elements $a,b,c$, and 
let $L$ be the cyclic group of order 3 with generator $\ell$, permuting
$a,b,c$ cyclically.
We work over a field $k$ of characteristic 2.
Let $U$ be the $kG$-module given by the quotient of the left
regular module $kG$ by the ideal generated by $a-1$.
Let $A=kG\ot k[L]$, and consider the following $A$-modules: 
$\ M=U\ot k_1$ and $N=k\ot k_{\ell}$.
The support varieties of $M\ot N$ and $N\ot M$ are 
\begin{eqnarray*}
     V_A(M\ot N) & = & 
      V_{G}(  U)\times \ell , \ \mbox{ and}\\
  V_A(N\ot M) &=& V_A(\hspace{.05cm}  {^{\ell}\! U} \ot k_{\ell}) \ \ = \ \
   {^{\ell}\! V_{G}( U)}\times \ell \ \ \neq \ \ V_{A}(M\ot N),
\end{eqnarray*}
since $V_G(U)\neq { ^{\ell}\! V_{G}(U)}$. 
For comparison, note that $V_A(M)\cap V_A(N) = \varnothing$.
We thus see that 
the variety of the tensor product depends on the order and 
is} not {\em contained in the intersection of the varieties.
The same is true even if we use the version of support varieties  in \cite{FW},
contradicting Proposition~2.4(5) in that paper; see
\cite{FW-E} for a correction. 

By modifying these modules, we obtain an example where $M\ot N$ is
projective while $N\ot M$ is not:
Let $M = U\ot k_{\ell}$ and $N=U\ot k_1$.
Then $M\ot N\cong (U\ot {^{\ell}\! U}) \ot k_{\ell}$, which has empty variety,
and so is projective. On the other hand, $N\ot M \cong (U\ot U)\ot k_{\ell}$,
which has (nonempty) variety corresponding to that of $U$. 

Finally, if $M= U\ot k_{\ell}$ as before, then
$$
   V_A(M\ot M) = V_{G} (U\ot {^{\ell}\! U})\times \ell^2 =
   (V_{G}(U)\cap V_{G}( {^{\ell}\! U}))\times \ell^2 = \varnothing .
$$
Therefore $M\ot M$ is projective while $M$ is not.  
Moreover, $V_A(M^*) = V_G( ^{\ell}\! U^*) \times \ell^{-1}$,
which differs from $V_A(M)$
 (and their 
support varieties as defined in \cite{FW}  also differ from each other). 
}
\end{example}

The projectivity of the second tensor power
of a nonprojective module is generalized in the next example.

\begin{example}{\em
Let $k$ be a field of positive characteristic $p$ and let
$n\geq 2$ be a positive integer.
Let  $G = (\Z/p\Z)^n$, generated by
$g_1,\ldots,g_n$, and 
$L= \Z/n\Z$, with generator $\ell$, acting on $G$ by cyclically permuting
these generators. 
Let $U$ be the $kG$-module that is the quotient of $kG$ by the ideal 
generated by $g_1-1, \ldots, g_{n-1} -1$, equivalently $U$ is the trivial
module induced from the subgroup $ \la g_1,\ldots, g_{n-1}\ra$. 
Then $$V_{G}( U\ot {^{\ell}\! U}\ot\cdots\ot {^{\ell^{n-2}}\! U})
= V_{G}( U)\cap V_{G}( {^{\ell}\! U})\cap \cdots
\cap V_{G}( {^{\ell ^{n-2}}\! U}) 
\neq \varnothing ,$$ 
while 
$$V_{G}(U\ot {^{\ell}\! U}\ot  \cdots \ot {^{\ell^{n-1}}\! U}) = \varnothing .
$$
Therefore, letting  $A=kG\ot k[L]$ and $M=U\ot k_{\ell}$, 
we have
$$
   V_A(M^{\ot (n-1)})\neq \varnothing \ \ \ \mbox{ and } \ \ \
   V_A(M^{\ot n}) =\varnothing .
$$
As a consequence, $M^{\ot n}$ is projective while $M^{\ot (n-1)}$ is not projective.
(Compare with \cite[Theorem~6.1]{DP}, in which a similar phenomenon occurs in the
unbounded derived category of a non-Noetherian commutative ring.)
}
\end{example}

\section{Localising subcategories}\label{localising}

Next, we turn to applications of support theory to the structure
of the 
stable module category $\StMod(A)$, a triangulated category 
whose objects are all $A$-modules
(including infinite dimensional ones),
and morphisms are all $A$-module homomorphisms modulo those factoring through
projective modules.  
For our purposes in this section, we require a version
of support varieties that works well for infinitely generated
modules and agrees with the definition given in the last
section for finitely generated modules. 

Set $R = \coh^*(G,k)\otimes k[L]$.
Recall that there is a canonical inclusion of $R$ into the
Hochschild cohomology ring $\HH^*(A)$.
Composing this with the natural map from $\HH^*(A)$ to the
graded center $Z^*(\StMod(A))$,
we obtain an action of $R$ on $\StMod(A)$ in the sense of 
\cite{Benson/Iyengar/Krause:2008a} by the first author, Iyengar, and Krause.
As in Section 5 of \cite{Benson/Iyengar/Krause:2008a}, 
we obtain local cohomology functors
$\Gamma_\p\colon \StMod(A)\to\StMod(A)$ for each $\p$ in the
homogeneous prime ideal spectrum of $R$. We have $\Gamma_\p \ne 0$ if and
only if $\p$ is not a maximal ideal. We define $\V_A$
to be the projective spectrum of nonmaximal homogeneous prime ideals
of $R$, and we define the {\em support} of an $A$-module $M$ to be 
 \[\V_A(M) = \{\p \mid \Gamma_\p(M) \ne 0\}.\]
In case $M$ is finite dimensional, $\mathcal V_A(M)$ coincides with
the set of nonmaximal homogeneous primes containing $I_A(M)$,
and therefore it determines and is determined by $V_A(M)$.

A {\em left ideal localising subcategory} of $\StMod(A)$ is a full
triangulated subcategory that is closed under taking direct summands,
arbitrary direct sums, and tensoring on the left by objects of 
$\StMod(A)$. (In fact, the latter two properties imply the first;
see  \cite{Benson/Iyengar/Krause:2011b}.)
Similarly define {\em right} and {\em two-sided ideal localising subcategories}. 

The main theorem of \cite{Benson/Iyengar/Krause:2011b} implies that 
if we regard $\StMod(kG)$ as a tensor triangulated category, it
is stratified by the action of $H^*(G,k)$. Loosely speaking, 
this means that the tensor ideal localising subcategories are
classified in terms of subsets of the variety $\V_G=\Proj H^*(G,k)$,
the projective spectrum of homogeneous
prime ideals of the cohomology ring $H^*(G,k)$, but not
including the maximal ideal of positive degree elements.
If we forget the tensor structure, this is still true provided 
that $G$ is a finite $p$-group, but not more generally.

Next we describe the classification theorem for tensor ideals
in $\StMod(A)$.
Since the tensor product formula in Theorem~\ref{tp-formula} gives
\[ (k\otimes k_{\ell_1}) \otimes (M_\ell \otimes k_\ell) \otimes (k
\otimes k_{\ell_2^{-1}}) = {^{\ell_1}}M_\ell \otimes
k_{\ell_1\ell\ell_2^{-1}}, \]
we introduce an action of $L\times L$ on $\V_G\times L$ as follows.
\[ (\ell_1,\ell_2)\colon (x,\ell) \mapsto
({^{\ell_1}}x,\ell_1\ell\ell_2^{-1}). \]

The next theorem is a straightforward consequence of the tensor product
formula  in Theorem~\ref{tp-formula} and 
the result  
\cite[Theorem~10.3]{Benson/Iyengar/Krause:2011b} of the first author,
Iyengar and Krause in the finite group setting. 

\begin{thm}
Let $A$ be the Hopf algebra described above. Then the 
theory of support gives the following bijections:
\begin{itemize}
\item[\rm (i)] The left ideal localising subcategories of $\StMod(A)$
and the subsets of the set of orbits of $L \times 1$ on $\V_G \times L$.
\item[\rm (ii)] The right ideal localising subcategories of
  $\StMod(A)$
and the subsets of the set of orbits of $1\times L$ on $\V_G \times L$.
\item[\rm (iii)] The two-sided ideal localising subcategories of
$\StMod(A)$ and the subsets of the set of orbits of $L\times L$ on $\V_G\times L$. 
\end{itemize}
\end{thm}

In each case the correspondence takes a set of orbits to the full subcategory
of modules whose support lies in one of the orbits in the set,
and it takes a subcategory to the union of the supports of its objects. 

\begin{remark}
Abstractly, the orbits in (i) and (ii) both give copies of $\V_G$,
but the left ideals are not the same as the right ideals. The
orbits in (iii) give the quotient $\V_G/L$, namely the projective spectrum
of the fixed points $H^*(G,k)^L$. 
\end{remark}

There are analogous statements for the category $\Mod(A)$ of all $A$-modules,
an application of \cite[Theorem~10.4]{Benson/Iyengar/Krause:2011b} for finite groups: 
The same sets of orbits classify non-zero tensor ideal localising
subcategories of $\Mod(A)$. (A full subcategory $\mathcal C$ of $\Mod(A)$ is
localising if direct sums and direct summands of modules in $\mathcal C$
are also in $\mathcal C$ and if any two modules in an exact sequence
of $A$-modules are in $\mathcal C$, then so is the third.
It is a tensor ideal if it is closed under tensor products with modules in $\Mod(A)$.) 
Similary, one obtains analogous classification results for the categories
of finite dimensional modules, $\mmod(A)$ and $\stmod(A)$.

\end{document}